\documentclass[11pt]{article}
  \usepackage{graphicx}
  \usepackage{amsmath,amssymb}
\usepackage{amsthm}

\newtheorem{theorem}{Theorem}[section]

\theoremstyle{definition}

\newtheorem{example}[theorem]{Example}
\newtheorem{proposition}[theorem]{Proposition}

\newtheorem{remark}[theorem]{Remark}

\numberwithin{equation}{section}

\begin{document}

\title{Attractors of  Caputo fractional  differential equations with  triangular vector fields}

\author{Thai Son Doan\footnote{Email: dtson@math.ac.vnInstitute of Mathematics, Vietnam Academy of Science and Technology, 18 Hoang Quoc Viet, Cau Giay, Ha Noi, Viet Nam}
 \quad and \quad Peter E. Kloeden
\footnote{Email:kloeden@math.uni-frankfurt.de, Mathematisches Institut, Universit\"at T\"ubingen, D-72076 T\"ubingen, Germany}
}



\date{\notag}
\maketitle


\begin{abstract}
It is shown that the attractor  of  an autonomous   Caputo fractional   differential equation  of order $\alpha\in(0,1)$  in $\mathbb{R}^d$ whose vector field  has a  certain triangular structure and  satisfies a smooth condition and dissipativity condition is essentially the same as that of the ordinary differential equation with the same vector field.  As an application, we establish several one-parameter bifurcations for scalar fractional differential equations including the saddle-node and the pichfork bifurcations. The proof uses a result of
Cong \&	Tuan \cite{cong} which shows that no two solutions of such a Caputo FDE can intersect in finite time.
\end{abstract}
{\footnotesize

}

\section{Introduction}
The asymptotic behaviour of Caputo fractional   differential equations (Caputo FDE) in $\mathbb{R}^d$  has attracted much attention in the literature in recent years. It has often been asked if such equations on $\mathbb{R}^d$, generate an   autonomous dynamical system, since that would allow the theory of attractors to applied to them.

Consider an  autonomous   Caputo FDE of order $\alpha\in(0,1)$ in $\mathbb{R}^d$  of the following form
\begin{equation}\label{acfde}
{}^{C}\!D^{\alpha}_{0+} x(t)= g(x(t))
\end{equation}
where $g : \mathbb{R}^d\rightarrow \mathbb{R}^{d}$  is    Lipschitz  continuous and satisfies a growth bound.  The  Caputo FDE \eqref{acfde} with the initial condition $x(0)$ $=$ $x_0$
is the  integral equation
\begin{equation}\label{AIE}
x(t)=x_0+\frac{1}{\Gamma(\alpha)} \int_0^t (t-s)^{\alpha -1}   g(x(s)) ds,
\end{equation}
where
  $\Gamma(\alpha):=\int_0^\infty t^{\alpha-1}e^{-t}dt$ is the Gamma function.

It is easy to show that the  ordinary differential equation (ODE) with the same  vector field, i.e.,
\begin{equation}\label{ode}
\frac{ d}{dt}x(t)= g(x(t)),
\end{equation}
has an attractor when the vector field satisfies a dissipativity condition such as for $a,b>0$
\begin{equation}\label{diss}
\left<x,g(x)\right> \leq a - b \|x\|^2.
\end{equation}
Specifically, by the chain rule  along a solution of \eqref{ode},
$$
\frac{ d}{dt}\|x(t)\|^2 = 2 \left<x(t),g(x(t))\right> \leq 2 a - 2 b\|x(t)\|^2,
$$
which integrates  to give
$$
\|x(t)\|^2 \leq \|x_0\|^2 e^{ - 2b  t} +  \frac{a}{b} \left(1 -e^{ - 2b  t}\right).
$$
Hence the set
$$
\mathcal{B} := \left\{ x  \in \mathbb{R}^d  \, :  \,  \|x\|^2  \leq 1 +  \frac{ a}{b} \right\}
$$
is an absorbing set for the autonomous semi-dynamical system generated by the solution mapping of the ODE \eqref{ode},  which is positive invariant. In particular, this means that
this system has a global attractor,
$$
\mathcal{A}  = \bigcap_{t \geq 0} x(t,\mathcal{B})=\Omega_{\mathcal{B}},
$$
where $x(t,\mathcal{B})=\cup_{\eta\in B} x(t,\eta)$ and $\Omega_{\mathcal{B}} $ is   the omega limit set defined by
$$
\Omega_{\mathcal{B}} := \left\{  y  \in\mathcal{B}  \, :  \,  \exists\, x_{0,n }\in\mathcal{B} ,  t_n \to \infty\, \, \mbox{such that} \,\,  x(t_n,x_{0,n})\to y   \right\}.
$$

\smallskip

By a recent result of   Aguila-Camacho {\em et al.} \cite[Lemma 1]{ACDMG} it is known that a  solution of  the Caputo  FDE   \eqref{acfde} satisfies
$$
{}^{C}\!D^{\alpha}_{0+} \|x(t)\|^2   \leq  2 \left< x(t)   ,  {}^{C}\!D^{\alpha}_{0+}  x(t)  \right>.
$$
Hence, if the vector field $g$ of \eqref{acfde} satisfies the dissipativity condition  \eqref{diss}, then along the solutions of \eqref{acfde}
$$
  {}^{C}\!D^{\alpha}_{0+} \|x(t)\|^2   \leq     2 \left<x(t),g(x(t))\right> \leq 2 a - 2b \|x(t)\|^2
$$
as in the ODE case. Then, by  Wang \& Xiao \cite[Theorem 1]{WX},  the corresponding set $\mathcal{B}$ defined  in terms of the solutions of the Caputo  FDE  is an absorbing set for the solutions of the  Caputo  FDE   \eqref{acfde}. Since this set is compact  in  $\mathbb{R}^d$,   the corresponding omega limit set $\Omega_{\mathcal{B}}$ exists and  is a nonempty compact subset of  $\mathcal{B}$, which attracts all of  the future dynamics of the Caputo  FDE.   It is clear that  $\Omega_{\mathcal{B}}$ contains all of the steady state solutions of  \eqref{acfde}.

In general, $\Omega_{\mathcal{B}}$ cannot be called the attractor of the autonomous  Caputo  FDE   \eqref{acfde}. Recently   Cong \& Tuan  \cite{cong} confirmed the conjecture in \cite{diethelm2008,diethelm2012} by showing  that the solution mapping of   a general autonomous Caputo  FDE  \eqref{acfde} on $\mathbb{R}^d$ does not generate a semi-group on $\mathbb{R}^d$ and, hence,  there is no
autonomous semi-dynamical system on   $\mathbb{R}^d$ corresponding to  \eqref{acfde}.  Consequently,  since solutions cannot be in general concatenated, there may also be omega limit points of solutions starting outside $\mathcal{B}$  that are not in  $\Omega_{\mathcal{B}}$. Also, strictly speaking, mathematically, a general Caputo  FDE   \eqref{acfde}   has    no attractor on  $\mathbb{R}^d$ since this concept is defined in terms of an autonomous semi-dynamical system.

However,  Cong \& Tuan  \cite{cong}   showed that  a Caputo  FDE   \eqref{acfde} with a triangular vector   field  does generate a semi-dynamical system on  $\mathbb{R}^d$.
Recall that a vector field $g : \mathbb{R}^d\rightarrow \mathbb{R}^{d}$ is called  triangular if it has components with the structure $
g_1(x_1),    g_2(x_1,x_2),    \cdots,    g_d(x_1,x_2,\cdots,x_d)$ which covers scalar vector fields as a special case. Our aim in this paper is to use the result in \cite{cong} to investigate the attractor of  Caputo fractional differential equations with a triangular vector field. The result for an attractor of scalar  Caputo fractional differential equations is presented in Section \ref{Section2}. A generalization to  Caputo fractional differential equations with a certain triangular vector field is presented in Section \ref{Section3}. Several examples of bifurcations of scalar Caputo FDEs are presented in Section \ref{Section4}. Section \ref{Section5} is devoted to discussing a potential approach to attractors of general Caputo FDEs by using the existing theory of attractors for semi-dynamical systems on function spaces.
\section{Attractor of scalar Caputo fractional differential equations}\label{Section2}

In this subsection, we consider following scalar fractional differential equation
\begin{equation}\label{ScalarCaputo}
{}^{C}\!D^{\alpha}_{0+} x(t)= g(x(t)),\qquad t\geq 0,
\end{equation}
where $g:\mathbb R\rightarrow \mathbb R$ is a continuously differentiable function and satisfies that
\begin{itemize}
\item [(H1)] (Dissipative condition): There exist $a,b>0$ such that
\[
g(x) x\leq a-bx^2\quad \hbox{ for all } x\in\mathbb R.
\]
\item [(H2)] (Non-degenerate condition):  $g'(x)\not=0$ for all $x\in\mathcal N(g):=\{x\in\mathbb R: g(x)=0\}$.
\end{itemize}
\begin{remark} By (H1), we have $\mathcal N(g)\subseteq [-\sqrt{\frac{a}{b}},\sqrt{\frac{a}{b}}]$. By (H2), $\mathcal N(g)$ has no accumulation point and therefore   $\mathcal N(g)$ has a finite elements and the number of elements is odd. Furthermore, let $\mathcal N(g)=\{x_1,\dots,x_{2k+1}\}$. Then, for all $i=0,1,\dots,k$
\begin{equation}\label{New_Eq1}
g(x)>0\quad \hbox{ for all }x \in (x_{2i},x_{2i+1}),
\end{equation}
and
\begin{equation}\label{New_Eq2}
g(x)<0\quad \hbox{ for all }x \in (x_{2i+1},x_{2i+2}),
\end{equation}
where we use the conventions that $x_0:=-\infty$ and $x_{2k+2}=\infty$.
\end{remark}
The main result of this section is the following theorem about  attractors for a scalar Caputo fractional differential equation \eqref{ScalarCaputo}.
\begin{theorem}[Attractor for scalar Caputo fractional differential equations]\label{MainTheorem1}
Consider system \eqref{ScalarCaputo}. Suppose that the assumptions (H1) and (H2) hold. Then, the following statements hold:
\begin{itemize}
\item [(i)] The global attractor attracting all solutions starting from a bounded set is
\[
\mathcal A=[\min \mathcal N(g), \max \mathcal N(g)].
\]
\item [(ii)] Each solution of \eqref{ScalarCaputo} converges to an element of $\mathcal N(g)$ and the rate of convergence is $t^{-\alpha}$.
\item [(iii)] Each pair of successive values of $\mathcal N(g)$ is a heteroclinic solution of \eqref{ScalarCaputo}
\end{itemize}
\end{theorem}
To prove the above theorem, we need several preparatory results. The following proposition indicates that the set $A=[x_1,x_{2k+1}]$ attracts all solutions of \eqref{ScalarCaputo}. Note that this set $A$  includes all steady states $x_1,\dots,x_{2k+1}$.\\

\noindent
For $\alpha,\beta\in (0,1)$ the Mittag-Leffler function $E_{\alpha,\beta}:\mathbb R\rightarrow \mathbb R$ is defined as
\[
E_{\alpha,\beta}(z):=\sum_{k=0}^{\infty}\frac{z^k}{\Gamma(\alpha k+\beta)},\qquad E_{\alpha}(z):=E_{\alpha,1}(z).
\]
Let $d(A,B)$ denote the distance between two subsets $A$ and $B$ of $\mathbb R$.
\begin{proposition}\label{Scalar_Proposition}
Let $B$ be a bounded set of $\mathbb R$. For any $t\geq 0$, let $x(t,B):=\{x(t,\eta): \eta\in B\}$. Then, there exists $\lambda>0$ such that
\[
d(x(t,B), [x_1,x_{2k+1}] )\leq E_{\alpha}(-\lambda t^{\alpha}) d(B, [x_1,x_{2k+1}])
\quad\hbox{ for all } t\geq 0.
 \]
Consequently,
\[
\lim_{t\to\infty} d(x(t,B), [x_1,x_{2k+1}] )=0.
\]
\end{proposition}
\begin{proof}
Due to the non-intersection of two trajectories of \eqref{ScalarCaputo}, we have \[
x(t,\eta)\in [x(t,\inf B),x(t,\sup B)]\qquad \hbox{for all } \eta\in B.
\]
Then, to conclude the proof it is sufficient to show that for all $\eta\in\mathbb R$ there exists $\gamma>0$ such that
\begin{equation}
d(x(t,\eta), [x_1,x_{2k+1}] )
\leq
E_{\alpha}(-\gamma t^{\alpha}) d(\eta, [x_1,x_{2k+1}]).
\end{equation}
By using the non-intersection of two trajectories of \eqref{ScalarCaputo} and the fact that $x_1,x_{2k+1}$ are steady state solutions, for all  $\eta\in [x_1,x_{2k+1}]$ we have 
\[
d(x(t,\eta), [x_1,x_{2k+1}] )= d(\eta, [x_1,x_{2k+1}])=0,
\]
which implies that  the preceding conclusion obviously holds.  Then, it is enough to deal with the case that $\eta<x_1$ and  use analogous arguments for the case $\eta>x_{2k+1}$. Choose and fix $\eta<x_1$ and to conclude the proof we will show that
\begin{equation}\label{Eq_Aim}
 |x(t,\eta)-x_1|
 \leq
 E_{\alpha}(-\gamma t^{\alpha})  |x_1-\eta| \quad \hbox{for } t\geq 0.
\end{equation}
for some $\gamma>0$. The proof of the preceding fact is divided into two steps:\\

\noindent
\textbf{Step 1}: Consider a new fractional differential equation
\begin{equation}\label{New_Eq1}
{}^{C}\!D^{\alpha}_{0+} y(t)=
f(y(t)),
\end{equation}
where $f:\mathbb R\rightarrow \mathbb R$ is defined as
\begin{equation}\label{Newfunction}
f(x):=g(x+x_1)\qquad\hbox{ for all } x\in\mathbb R.
\end{equation}
Then, we show that $x(t,\eta)=x_1+y(t,\eta-x_1)$, where $y(\cdot,\zeta)$ denotes the solution of \eqref{New_Eq1} satisfying $y(0)=\zeta$. To see that, the integral form of \eqref{ScalarCaputo} yields that
\[
x(t,\eta)=\eta+\frac{1}{\Gamma(\alpha)}\int_0^t (t-s)^{\alpha-1}g(x(s,\eta))\;ds.
\]
Thus,
\begin{eqnarray*}
x(t,\eta)-x_1
&=&
\eta-x_1+\frac{1}{\Gamma(\alpha)}\int_0^t (t-s)^{\alpha-1}g(x(s,\eta)-x_1+x_1)\;ds\\
&=&\eta-x_1+\frac{1}{\Gamma(\alpha)}\int_0^t (t-s)^{\alpha-1}f(x(s,\eta)-x_1)\;ds,
\end{eqnarray*}
which implies that $y(t,x_1-\eta)=x(t,\eta)-x_1$. So, to prove \eqref{Eq_Aim} it is sufficient to show that
\begin{equation}\label{Eq_AimII}
|y(t,\zeta)|
\leq
E_{\alpha}(-\gamma t^{\alpha}) |\zeta| \qquad\hbox{ for  }\zeta:=x-x_1<0.
\end{equation}

\noindent
\textbf{Step 2}:
For this purpose, we first show that there exists $\gamma>0$ such that
\begin{equation}\label{Eq_01}
f(x)\geq \gamma |x|\qquad\hbox{ for all } x\leq 0.
\end{equation}
Indeed, by (H1) we have $g(x)>0$ for $x<x_1$ and therefore by (H2), $g'(x_1)<0$. Equivalently, by \eqref{Newfunction} we have $f(x)>0$ for all $x<0$, $f(0)=0$ and $f'(0)<0$. Hence,  by mean value theorem, there exists $\varepsilon >0$ such that
\[
|f(x)|\geq \frac{|f'(0)|}{2}|x|\qquad\hbox{for all } x\in [-\epsilon,\epsilon].
\]
Hence, if $x\geq -\epsilon$ then \eqref{Eq_01} holds for $\gamma:=\frac{|f'(0)|}{2}$. In the other case, i.e. $x<  -\epsilon$, let
\[
\gamma:=\min\left\{ \frac{|f'(0)|}{2}, \min_{w\in [\zeta,-\epsilon]} \frac{f(w)}{|w|}\right\}.
\]
Then, by strictly positivity of $f$ on $[\zeta,-\epsilon]$ we have $\gamma>0$ and obviously \eqref{Eq_01} also holds for this choice of $\gamma$. So, in both cases there exists $\gamma>0$ satisfying \eqref{Eq_01}. We now rewrite \eqref{New_Eq1} in the following form
\[
{}^{C}\!D^{\alpha}_{0+} y(t)= -\gamma y(t)+h(y(t)),
\]
where $h:\mathbb R\rightarrow \mathbb R$ is defined by
\[
h(y):=f(y)+\gamma y.
\]
On the one hand, by \eqref{Eq_01} we have
\begin{equation}\label{Eq_03}
h(y)\geq 0\qquad\hbox{for all } y\leq 0.
\end{equation}
Thanks to the variation of constants formula (see e.g. \cite[Lemma 3.1]{cong})  we arrive at the following representation of the solution $y(t,\zeta)$ as
\begin{equation}\label{Eq_04}
y(t,\zeta)=E_{\alpha}(-\gamma t^{\alpha})\zeta
+\frac{1}{\Gamma(\alpha)}\int_0^{t}
(t-s)^{\alpha-1} E_{\alpha,\alpha}(-\gamma (t-s)^{\alpha}) h(y(s,\zeta))\;ds.
\end{equation}
By the non-intersection of two solutions of \eqref{New_Eq1} we have $y(t,\zeta) \in [\zeta,0]$ for all $t\geq 0$. This together with \eqref{Eq_03} gives that
\[
h(y(s,\zeta))\geq 0\qquad\hbox{ for all } s\geq 0.
\]
Consequently, by \eqref{Eq_04} and positivity of the function $E_{\alpha,\alpha}$ we arrive at
\[
y(t,\zeta)\in [E_{\alpha}(-\gamma t^{\alpha})\zeta,0]\qquad\hbox{for all } t\geq 0,
\]
which  shows \eqref{Eq_AimII}. Furthermore, since $\lim_{t\to\infty}  E_{\alpha}(-\gamma t^{\alpha})=0$ it follows that $\lim_{t\to\infty} y(t,\zeta)=0$. The proof is complete.
\end{proof}
In the following result, we establish the asymptotic behavior of solutions starting inside the attractor. The idea of the proof of this proposition is quite similar to Proposition \ref{Scalar_Proposition} and we only sketch the main points of the proof.
\begin{proposition}\label{PoitwiseConvergence}
 The following statements hold:
\begin{itemize}
\item [(i)] For $i=0,\dots,k$ and $\eta\in (x_{2i},x_{2i+1})$ there exists $\gamma>0$ such that $|x(t,\eta)-x_{2i+1}|\leq E_{\alpha}(-\gamma t^{\alpha})|\eta-x_{2i+1}|$. Consequently, $\lim_{t\to\infty} x(t,\eta)=x_{2i+1}$.
\item [(ii)] For $i=0,\dots,k$ and $\eta\in (x_{2i+1},x_{2i+2})$  there exists $\gamma>0$ such that $|x(t,\eta)-x_{2i+1}|\leq E_{\alpha}(-\gamma t^{\alpha})|\eta-x_{2i+1}|$. Consequently, $\lim_{t\to\infty} x(t,\eta)=x_{2i+1}$.
\end{itemize}
\end{proposition}
\begin{proof}
We only give a proof of (i) and by using analogous arguments we also obtain (ii). Let $i\in \{0,1,\dots,k-1\}$ be arbitrary but fixed. From
\eqref{New_Eq1} and \eqref{New_Eq2}, we have
\begin{equation}\label{New_Eq5}
g(x)>0\quad \hbox{ for all } x\in (x_{2i},x_{2i+1})\quad
\hbox{ and }
g'(x_{2i+1})<0.
\end{equation}
Now, choose and fix an arbitrary $\eta\in (x_{2i},x_{2i+1})$. Consider a new fractional differential equation
\begin{equation}\label{New_Eq3}
{}^{C}\!D^{\alpha}_{0+} y(t)=
f(y(t)),
\end{equation}
where $f:\mathbb R\rightarrow \mathbb R$ is defined as
\[
f(y):=g(y+x_{2i+1})\qquad\hbox{ for all } x\in\mathbb R.
\]
Then, for $y(t,
\zeta)$ denoting the solution of \eqref{New_Eq3} we have $x(t,\eta)=y(t,\eta-{x_{2i+1}})+x_{2i+1}$ for all $t\geq 0$. Then it is sufficient to show that for all $\zeta \in (x_{2i}-x_{2i+1},0)$
\begin{equation}\label{New_Eq6}
|y(t,\zeta)|\leq E_{\alpha}(-\gamma t^{\alpha})|\zeta|\qquad\hbox{ for some } \gamma>0.
\end{equation}
Now, the property \eqref{New_Eq5} is translated into the function $f$ as
\[
f(y)>0\quad \hbox{ for all } y\in (x_{2i}-x_{2i+1},0)\quad
\hbox{ and }
f'(0)<0,
\]
which gives that there exists $\gamma>0$ (depending on $\zeta\in (x_{2i}-x_{2i+1},0)$) such that $f(y)\geq \gamma |y|$ for all $y\in [\zeta,0]$. Thus, by variation of constants formula we have
\[
y(t,\zeta)\geq E_{\alpha}(-\gamma t^{\alpha})\zeta \qquad\hbox{for all } t\geq 0,
\]
proving \eqref{New_Eq6}. The proof is complete.
\end{proof}

Next, we discuss the existence of heteroclinic orbits joining the steady state solutions. Here, we need to discuss how to define the value of solution in the negative time axis. Roughly speaking, we can extend the solution $x(\cdot,\eta)$ in the negative time axis as follows: for any $t\leq 0$ then  $x(t,\eta)$ is the unique value $\zeta\in\mathbb R$ satisfying that $x(-t,\zeta)=\eta$, it means that
\[
x(-t,x(t,\eta))=\eta.
\]
The well-defined property of this way of extension is confirmed by the result in \cite[Theorem 4.8]{cong}.

\begin{proposition}[Heteroclinic trajectory joining the steady states]\label{Heteroclinic}
 The following statements hold:
\begin{itemize}
\item [(i)] For $i=0,\dots,k-1$ and $\eta\in (x_{2i},x_{2i+1})$ the solution $x(t,\eta)$ is a heteroclinic trajectory joining  the steady states $x_{2i}$ and $x_{2i+1}$. More precisely, $\lim_{t\to-\infty} x(t,\eta)=x_{2i}$
 and $\lim_{t\to\infty} x(t,\eta)=x_{2i+1}$.
\item [(ii)] For $i=0,\dots,k$ and $\eta\in (x_{2i+1},x_{2i+2})$ the solution $x(t,\eta)$ is a heteroclinic trajectory joining  the steady states $x_{2i+1}$ and $x_{2i+2}$. More precisely, $\lim_{t\to-\infty} x(t,\eta)=x_{2i+2}$
 and $\lim_{t\to\infty} x(t,\eta)=x_{2i+1}$.
\end{itemize}
\end{proposition}
\begin{proof}
We only give a proof for the part (i) and refer an analogous argument for the proof of part (ii). In fact, by Proposition \ref{PoitwiseConvergence} it is only required to prove that
\begin{equation}\label{New_Eq7}
\lim_{t\to-\infty} x(t,\eta)=x_{2i}\qquad\hbox{ for all } \eta\in (x_{2i},x_{2i+1}).
\end{equation}
Analog to the proof of Proposition \ref{Scalar_Proposition} and Proposition \ref{PoitwiseConvergence}(i), we can introduce the new system to have the property that $x_{2i+1}=0$. So, in what follows we can assume additionally that $x_{2i+1}=0$. Choose and fix $\eta\in (x_{2i},0)$. We divide the remaining proof into several steps:\\

\noindent
\textbf{Step 1}: We show that there exists $\gamma>0$ such that
\begin{equation}\label{Step1}
g(\zeta)\geq \gamma (\zeta-x_{2i})|\zeta|\qquad\hbox{ for all } \zeta\in [x_{2i},0].
\end{equation}
To prove this, since $g'(x_{2i})>0>g'(0)$ it follows that there exists $\varepsilon \in (0,-\frac{x_{2i}}{3})$ such that
\[
g(\zeta)\geq \frac{g'(x_{2i})}{2} (\zeta-x_{2i})\qquad\hbox{ for all } \zeta \in (x_{2i},x_{2i}+\epsilon).
\]
and
\[
g(\zeta)\geq \frac{|g'(0)|}{2} |\zeta|\qquad\hbox{ for all } \zeta \in (-\epsilon,0).
\]
Then, \eqref{Step1} holds for
\[
\gamma:=\min\left\{
\frac{g'(x_{2i})}{2|x_{2i}|},\frac{|g'(0)|}{2|\epsilon+x_{2i}|}, \min_{\zeta\in [x_{2i}+\epsilon, -\epsilon]} \frac{g(\zeta)}{(\zeta-x_{2i})|\zeta|}
\right\}.
\]
The positivity of $\gamma$ follows from the fact that $g(\zeta)>0$ for all $\zeta\in [x_{2i}+\epsilon, -\epsilon]$.
\\

\noindent
\textbf{Step 2}: For any $\zeta \in (x_{2i},0)$, we show that
\begin{equation}\label{Step2}
x(t,\zeta)
\geq
E_{\alpha}(-\gamma (\zeta-x_{2i}) t^{\alpha})\zeta\qquad\hbox{ for all } t\geq 0.
\end{equation}
To show this inequality, choose and fix  $\zeta \in (x_{2i},0)$ and let $\widehat\gamma:= \gamma (\zeta-x_{2i})$. Then, we can write the Caputo fractional differential equation \eqref{ScalarCaputo} as
\[
{}^{C}\!D^{\alpha}_{0+} x(t)= -\widehat\gamma x(t) + \left(g(x(t))+\widehat\gamma x(t)\right).
\]
By the variation of constants formula, we have
\begin{eqnarray*}
x(t,\zeta)
&=& E_{\alpha}(-\widehat\gamma t^{\alpha}) \zeta\\[1ex]
&& +\frac{1}{\Gamma(\alpha)}\int_0^t (t-s)^{\alpha-1}E_{\alpha,\alpha} (-\widehat\gamma(t-s)^{\alpha}) (g(x(s,\zeta))+\widehat\gamma x(s,\zeta))\;ds.
\end{eqnarray*}
Since $x(s,\zeta)\geq \zeta$ for all $s\geq 0$ it follows with \eqref{Step2} that
\[
g(x(s,\zeta))+\widehat\gamma x(s,\zeta)\geq \gamma (\zeta-x_{2i}) |x(s,\zeta)|
+
\widehat\gamma x(s,\zeta)
\geq 0.
\]
Thus, $x(t,\zeta)
\geq E_{\alpha}(-\widehat\gamma t^{\alpha}) \zeta$ and \eqref{Step2} is proved.\\

\noindent
\textbf{Step 3}: Let $\eta\in (x_{2i},0)$ be arbitrary. Then,
\[
x(-t,x(t,\eta))=\eta\qquad\hbox{ for all } t<0,
\]
which together with \eqref{Step2} implies that
\begin{eqnarray*}
\eta
&\geq &
E_{\alpha}(-\gamma (x(t,\eta)-x_{2i}) (-t)^{\alpha})x(t,\eta)\quad\hbox{ for all } t<0\\[1ex]
&\geq &
E_{\alpha}(-\gamma (x(t,\eta)-x_{2i}) (-t)^{\alpha})x_{2i}\quad\hbox{ for all } t<0
\end{eqnarray*}
Since $E_{\alpha}(\cdot)$ is a montononically increasing function and
$$
\lim_{t\to -\infty}E_{\alpha}(-\rho(-t)^{\alpha})=0\quad\hbox{ for all } \rho>0
$$
it follows that $\lim_{t\to-\infty} x(t,\eta)=x_{2i}$. The proof is complete.
\end{proof}
We are now in a position to prove the main result of this section.
\begin{proof}[Proof of Theorem
\ref{MainTheorem1}]
The proof of (i) and (iii) are given in Proposition \ref{Scalar_Proposition} and Proposition \ref{Heteroclinic}, respectively. The first statement in (ii) that each solution of \eqref{ScalarCaputo} converges to an element of $\mathcal N(g)$ is given in Proposition \ref{PoitwiseConvergence}. It remains to show the rate of convergence. In fact, by using Proposition \ref{PoitwiseConvergence} for any $\eta$ there exists $\gamma>0$ such that
\begin{equation}\label{Estimate1}
d(x(t,\eta),\mathcal N(g))
\leq
E_{\alpha}(-\gamma t^{\alpha}) d(\eta,\mathcal N(g))\qquad\hbox{ for all } t\geq 0.
\end{equation}
On the other hand, by \cite[Theorem 4.1]{cong} there exists $L>0$ such that
\begin{equation}\label{Estimate2}
d(x(t,\eta),\mathcal N(g))
\geq
E_{\alpha}(-L t^{\alpha}) d(\eta,\mathcal N(g))\qquad\hbox{ for all } t\geq 0.
\end{equation}
Furthermore, for any $\lambda>0$ we have $\lim_{t\to\infty} t^{\alpha} E_{\alpha}(-\lambda t^{\alpha})$ is finite. Then, by using \eqref{Estimate1} and \eqref{Estimate2} the rate of convergence of any solution of \eqref{ScalarCaputo} to the steady states of \eqref{ScalarCaputo} is $t^{-\alpha}$.
\end{proof}
\begin{remark}[Comparison to the proof in the scalar ordinary differential equations]
Consider a scalar ordinary differential equation
\begin{equation}\label{ScalarODE}
\dot x(t)=g(x(t))
\end{equation}
Let $a,b$ be two successive zeros of $g$, i.e. $g(a)=g(b)=0$ and $g(x)\not=0$ for $x\in (a,b)$. Then, by continuity of $g$ either $g(x)>0$ for all $x\in (a,b)$ or either $g(x)<0$ for all $x\in (a,b)$. Then, any solution starting from a value in $(a,b)$ will be either strictly monotonically increasing or strictly monotonically decreasing. Consequently, any solution of \eqref{ScalarODE}  will converges to one of two steady states $a,b$.

\noindent
The above monotonicity argument of ODEs cannot extend to FDEs with the same vector field. The main reason is the appearance of the singular kernel in the integral form
\[
x(t,\eta)=\eta+\frac{1}{\Gamma(\alpha)}\int_0^t (t-s)^{\alpha-1} g(x(s,\eta))\;ds.
\]
\end{remark}
\begin{example}
Two specific  scalar Caputo FDE , i.e., with a vector field  $g$   $:$ $\mathbb{R}^1$  $\rightarrow$ $\mathbb{R}^1$, namely  will be investigated.
  $$
 g(x)=-x,\quad g(x) =   x - x^3.
  $$
  These satisfy a dissipativity condition and have steady state solutions $0$ and $0$, $\pm 1$, respectively.  The ODEs
  \[
  \frac{ d}{dt}x(t)= g(x(t))
  \]
  with these vector fields have global attractors
  $\mathcal A=\{0\}$ and $\mathcal A=[-1,1]$, respectively. Then, the corresponding Caputo FDEs
  \[
{}^{C}\!D^{\alpha}_{0+} x(t)= g(x(t))
  \]
  have the same steady state solutions  and  attractors, see Figure 1.  A major difference is that attraction or repulsion
of the  steady state solutions  is not at an exponential rate in the Caputo case.   Also, in the second example,  the heteroclinic trajectories joining the steady state solutions have the same geometric image in   $\mathbb{R}^{1}$, but different functional representations in the ODE and Caputo systems.
\begin{figure}
\begin{center}
 \includegraphics{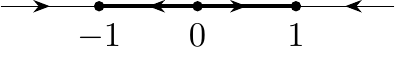}
\caption{Attractor $\mathcal A=[-1,1]$ for  $g(x) =   x - x^3$.}
\end{center}
\end{figure}
\end{example}

 \section{Attractor of Caputo fractional differential equations with triangular vector fields}\label{Section3}
In this section, we first generalize the result in previous section to a special class of Caputo fractional differential equations with triangular vector fields of the following form
\begin{equation}\label{CaputoFDE_Triangularform}
{}^{C}\!D^{\alpha}_{0+} x(t)= g(x(t))=(g_1(x(t)),\dots,g_d(x(t)))^{\mathrm T},
\end{equation}
where for $i=1,\dots,d$ we assume that the function $g_i:\mathbb R^{d}\rightarrow \mathbb R$ is of the following form
\[
g_i(x)=h_i(x_1,\dots,x_{i-1}) f_i(x_i)\qquad \hbox{ for } i=1,\dots,d.
\]
The function $g$ is assumed to be continuously differentiable and to satisfy the following hypothesises

\begin{itemize}
\item [(H1)] (Dissipative condition): There exist $a,b>0$ such that
\[
\langle x, g(x)\rangle \leq a-b \|x\|^2\quad \hbox{ for all } x\in\mathbb R^d.
\]
\item [(H2)] (Non-degenerate condition):  $g'_{ii}(u)\not=0$ for all $u\in\mathcal N(g_{ii}):=\{u\in\mathbb R: g_{ii}(u)=0\}$.
\end{itemize}
Since the structure of vector field in \eqref{CaputoFDE_Triangularform} is of product form it follows with the assume (H1) that for all $i=1,\dots,d$ the function $h_i:\mathbb R^{i-1}\rightarrow \mathbb R$ does not vanishing. Thus, the sign of $g_i$ depends only the scalar function $f_i$. So, applying Theorem \ref{MainTheorem1} to every components leads to the following result.
\begin{theorem}[Attractor for Caputo fractional differential equations of triangular vector fields]\label{MainTheorem2}
Consider system \eqref{CaputoFDE_Triangularform}. Suppose that the assumptions (H1) and (H2) hold. Then, the following statements hold:
\begin{itemize}
\item [(i)] The global attractor attracting all solutions starting from a bounded set is
\[
\mathcal A=[\min \mathcal N(g_{11}), \max \mathcal N(g_{11})]\times \dots\times [\min \mathcal N(g_{dd}), \max \mathcal N(g_{dd})].
\]
\item [(ii)] Each solution of \eqref{ScalarCaputo} converges to an element of the following set
$$
\mathcal N(g)
:=\mathcal N(g_{11})\times\dots\times\mathcal N(g_{dd})
$$ and the rate of convergence is $t^{-\alpha}$.
\end{itemize}
\end{theorem}
\begin{example}
Consider the following FDEs with the following vector fields
\begin{eqnarray*}
{}^{C}\!D^{\alpha}_{0+} x(t)
&=& x(t)(1-x(t)),\\[1ex]
{}^{C}\!D^{\alpha}_{0+} y(t)
&=& y(t)(1-y(t)^2) (1+x(t)^2).
\end{eqnarray*}
So, the attractor for the above equation is $\mathcal A=[-1,1]\times [-1,1]$. The asymptotically behavior of solutions are depicted in the Figure 2.

\begin{figure}\label{Fig2}
\begin{center}
\includegraphics[scale=0.9]{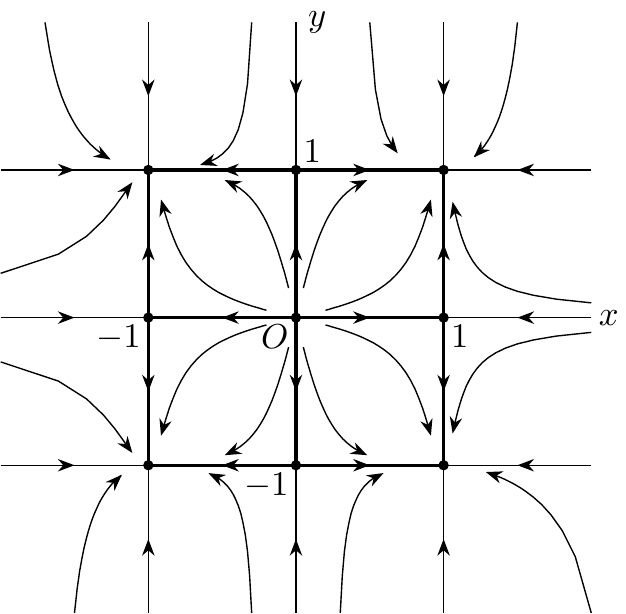}
\caption{Attractor $\mathcal A=[-1,1] \times [-1,1] $ with steady state and  heteroclinic  trajectories for the vector field    $g(x)=x(1 - x^2), f(x,y)$ = $y(1 - y^2)(1 +x^2)$.}
\end{center}
\end{figure}

\end{example}
%
%
%
%
\begin{remark}
It is interesting to know whether Theorem \ref{MainTheorem2} remains true when the vector field
is of a more general form of triangular vector fields, e.g.
\begin{eqnarray*}
{}^{C}\!D^{\alpha}_{0+} x(t)
&=& f(x(t)),\\[1ex]
{}^{C}\!D^{\alpha}_{0+} y(t)
&=& g(x(t),y(t)).
\end{eqnarray*}
Note the result of non-intersection of two solutions is still true for this equation, see \cite{cong}. However, the sign of the vector field $g(x(t),y(t))$ depends on both $x(t)$ and $y(t)$ and we can not use the approach in Section 2 to this problem.
\end{remark}

\section{One-parameter bifurcations for scalar Caputo fractional differential equations}\label{Section4}
Consider a family of scalar Caputo FDEs \eqref{ScalarCaputo}
\begin{equation}\label{Bifur_01}
{}^{C}\!D^{\alpha}_{0+} x(t)=g(\gamma,x),
\end{equation}
where $\gamma$ is a parameter. Let $x_{\gamma}(t,\eta)$ denote the solution of \eqref{Bifur_01} satisfying $x(0)=\eta$. It follows from Theorem \ref{MainTheorem1} that \eqref{Bifur_01} has the same bifurcations as an ODE  with the same vector field. This means, in particular, that the simpler steady states and sign of vector fields for the ODE can be used to determine the bifurcations of the corresponding  Caputo FDE. In what follows, we study the saddle-node and pitchfork bifurcations for fractional differential equations. We refer the readers to \cite[Section 2.1]{Hale} for a corresponding bifurcation analysis of ODE.
\begin{example}[Saddle-node bifurcation] Consider the following family of scalar Caputo FDEs
\begin{equation}\label{Bifur_02}
{}^{C}\!D^{\alpha}_{0+} x(t)=\gamma-x(t)^2.
\end{equation}
Then, the following statements hold:
\begin{itemize}
\item [(i)] For $\gamma<0$, then all solutions of \eqref{Bifur_02} tends to $-\infty$.
\item [(ii)] For $\gamma\geq 0$, then \eqref{Bifur_02} has two steady states $x=-\sqrt{\gamma}$ and $x=\sqrt{\gamma}$ and
\[
\lim_{t\to\infty}x_{\gamma}(t,\eta)=
\left\{
  \begin{array}{ll}
    -\infty, & \hbox{if }  \eta< -\sqrt{\gamma};\\[1ex]
    \sqrt{\gamma}, & \hbox{if } \eta>-\sqrt{\gamma}.
  \end{array}
\right.
\]
\end{itemize}
An analytical proof for (i) comes from the fact that
\[
x_{\gamma}(t,\eta)=\eta+\frac{1}{\Gamma(\alpha)}\int_0^t (\gamma-x_{\gamma}(s,\eta)^2)\;ds\leq \eta+ \frac{\gamma t}{\Gamma(\alpha)}.
\]
For the case $\gamma\geq 0$, an analogous argument as in (i) implies that $\lim_{t\to\infty}x_{\gamma}(t,\eta)=-\infty$ for  $\eta< -\sqrt{\gamma}$. Meanwhile, using Theorem \ref{MainTheorem1} for the restriction of \eqref{Bifur_02} on $(-\sqrt{\gamma},\infty)$ leads to $\lim_{t\to\infty}x_{\gamma}(t,\eta)=\sqrt{\gamma}$ if $\eta> -\sqrt{\gamma}$.
\end{example}
\begin{example}[Pitchfork bifurcation] Consider the following family of scalar Caputo FDEs
\begin{equation}\label{Bifur_03}
{}^{C}\!D^{\alpha}_{0+} x(t)=\gamma x(t)-x(t)^3.
\end{equation}
Then, by Theorem \ref{MainTheorem1} we obtain the following description of the bifurcation of the asymptotical behavior of solutions of \eqref{Bifur_03} on the parameter $\gamma$:
\begin{itemize}
\item [(i)] For $\gamma<0$, then \eqref{Bifur_02} has a steady state $x=0$ attracting all solutions of \eqref{Bifur_02}.
\item [(ii)] For $\gamma\geq 0$, then \eqref{Bifur_02} has three steady states $x=-\sqrt{\gamma}, x=0, x=\sqrt{\gamma}$ and
\[
\lim_{t\to\infty}x_{\gamma}(t,\eta)=
\left\{
  \begin{array}{ll}
    -\sqrt{\gamma}, & \hbox{if }  \eta< 0;\\[1ex]
    \sqrt{\gamma}, & \hbox{if } \eta>0.
  \end{array}
\right.
\]
\end{itemize}
\end{example}
\section{Attractors of Caputo semi-dynamical systems: the general case}\label{Section5}
Doan \& Kloeden \cite{doanklo} showed recently that  a  general autonomous Caputo fractional   differential equation
\begin{equation}\label{GeneralFDE}
{}^{C}\!D^{\alpha}_{0+} x(t)
= g(x(t)),\qquad \hbox{ where } g:\mathbb R^d\rightarrow \mathbb R^d,
\end{equation}
generates a semi-dynamical system on the  function space   $\mathfrak{C}$   of continuous functions $f:\mathbb{R}^+\rightarrow \mathbb{R}^d$   with the topology uniform convergence on compact subsets.  This topology is induced by the  metric
\[
\rho(f,h):=\sum_{n=1}^{\infty}\frac{1}{2^n} \rho_n(f,h),\quad \hbox{where }
\rho_n(f,h):=\frac{\sup_{t\in[0,n]}\|f(t)-h(t)\|}{1+\sup_{t\in[0,n]}\|f(t)-h(t)\|}.
\]
Define the operators $T_\tau$ $:$ $\mathfrak{C}$ $\rightarrow$ $\mathfrak{C}$, $\tau$ $\in$ $\mathbb{R}^+$,  by
\begin{equation}\label{Semigroup}
(T_{\tau} f)(\theta) =f(\tau + \theta)+   \frac{1}{\Gamma(\alpha)} \int_0^{\tau} (\tau+\theta-s)^{\alpha -1}  g(x_f(s))\;ds,  \qquad \theta \in \mathbb{R}^+,
\end{equation}
where $x_f$ is a solution of the  singular  Volterra integral equation   for this $f$, i.e.,
\begin{equation}\label{SVIE}
x_f(t)=f(t)+  \frac{1}{\Gamma(\alpha)} \int_0^t (t-s)^{\alpha-1}  g(x_f(s))\;ds.
\end{equation}
It was shown by Doan \& Kloeden \cite{doanklo} that the operators $T_\tau$, $\tau$ $\in$ $\mathbb{R}^+$, form a semi-group  on the space $\mathfrak{C}$.
(The proof in  \cite{doanklo}  follows  Chapter XI,  pages 178-179,  in Sell \cite{sell} closely). This semi-group  represents the  Caputo FDE \eqref{GeneralFDE}
as an autonomous semi-dynamical system on the space $\mathfrak{C}$.

The theory of autonomous semi-dynamical systems (see e.g.,  \cite{KR}) can be applied to the Caputo semi-group defined above.
 \begin{theorem}\label{thmattr}
	Suppose that  the    semi-dynamical system   $\{T_\tau, \tau \in \mathbb{R}^+\}$ on the space $\mathfrak{C}$ has a closed and bounded positively invariant absorbing set
	$\mathfrak{B}$ in $\mathfrak{C}$ and is  asymptotically compact. Then  the semi-dynamical system $\{T_\tau, \tau \in \mathbb{R}^+\}$ has a global attractor given by
	$$
	\mathfrak{A} = \bigcap_{t \geq 0} T_t(\mathfrak{B}).
	$$
\end{theorem}
The solution $x(t,x_0)$  of the  autonomous  Caputo FDE  \eqref{GeneralFDE} on  $\mathbb{R}^d$   corresponds to a  constant function $f_0(t)$ $\equiv$ $x_0$ and
$$
x(t,x_0)  \equiv (T_t f_0)(0).
$$

Thus, when the semi-group    $\{T_\tau, \tau \in \mathbb{R}^+\}$ has an attractor  $\mathfrak{A}$ $\subset$   $\mathfrak{C}$,   then   an  omega limit point $x$ $\in$ $\mathbb{R}^d$  of   trajectories of the Caputo FDE  satisfies $x$ $=$ $f(0)$ for some   function $f$  $\in$   $\mathfrak{A}$. In particular, if $g(x^*)$ $=$ $0$, then $f^*$  $\in$   $\mathfrak{A}$ for the constant function $f^*(t)$   $\equiv$ $x^*$, i.e.,  $x^*$ is a steady  state solution of the system. But   there may be   functions $f^*$  $\in$   $\mathfrak{A}$  that are not constant functions, so the strict inclusion,  $\Omega_{\mathcal{B}}\subsetneq\mathfrak{A}(0)$ usually holds, where $\Omega_{\mathcal{B}}$  the     omega limit  point set discussed in the introduction section and $\mathfrak{A}(0)$  is the set of values  in  $\mathbb{R}^d$ of the functions in $\mathfrak{A}$ when evaluated at $t$ $=$ $0$.

The application of Theorem \ref{thmattr} requires  determining an absorbing set  $\mathfrak{B}$ in $\mathfrak{C}$  and showing that the semi-dynamical system   $\{T_\tau, \tau \in \mathbb{R}^+\}$ is asymptotically compact in some sense. This will be investigated in another paper.



\end{document}